\documentclass[11pt, reqno]{amsart}
\usepackage{amsfonts,amssymb,amscd,amsmath,mathrsfs,amsthm}

\newcommand{\C}{\mathbb{C}}
\newcommand{\R}{\mathbb{R}}

\newcommand{\Z}{\mathbb{Z}}
\newcommand{\mc}[1]{\mathcal{#1}}
\newcommand{\g}{\mathfrak{g}}

\newtheorem{lemma}{Lemma}
\newtheorem{theorem}{Theorem}

\newtheorem{corollary}{Corollary}

\theoremstyle{definition}

\author{Friedrich Littmann}
\address{North Dakota State University, Department of Mathematics, Fargo, ND 58108-6050}
\email{Friedrich Littmann@ndsu.edu}
\author{Mark Spanier}
\address{Dakota State University, The Beacom College of Computer and Cyber Science, Madison, SD 57042 }
\email{Mark.Spanier@dsu.edu}

\title[Entire Gr\"unwald operators]{Weighted uniform convergence of entire Gr\"unwald operators on the real line}

 
\begin{document}
  
\maketitle

\begin{abstract}  We consider weighted uniform convergence of entire analogues of the Gr\"unwald operator on the real line.  The main result deals with convergence of entire interpolations of exponential type $\tau>0$ at zeros of Bessel functions in spaces with homogeneous weights. We discuss extensions to Gr\"unwald operators from de Branges spaces.
\end{abstract}

\section{Introduction and Results}
 
An entire function $F$, not identically zero, has exponential type if $\tau(F)$ defined by
\begin{align}\label{intro-type}
\tau(F) = \limsup_{|z|\to \infty} |z|^{-1}\log|F(z)|
\end{align}
is finite. The nonnegative number $\tau(F)$ is called the exponential type of $F$.

Let $w$ be a measurable, nonnegative function on $\R$; we call $w$ a weight.  We denote by $\mc{B}_p(\tau,w)$ the space of entire functions $F$ of exponential type $\tau\ge 0$ with $Fw\in L^p(\R)$. For functions $f:\R \to \C$ and a weight $w$, we seek discrete sets $\mc{T}\subseteq \R$  and $G_\tau f\in \mc{B}_\infty(\tau,w)$ with $G_\tau f(t) = f(t)$ for $t\in\mc{T}$  and
\begin{align}\label{G-conv}
\lim_{\tau\to\infty} \|(G_\tau f-f) w\|_\infty =0.
\end{align}

  By way of motivation we review known results from polynomial interpolation.  Fej\'er discovered the following property of interpolation at the zeros of the Chebyshev polynomials: denoting by $x_{n,k}$ the $k$th zero of the $n$th Chebyshev polynomial, there exists  a polynomial $H_{2n-1}$ of degree $2n-1$ with $H_{2n-1}(x_{n,k}) = f(x_{n,k})$ and $H_{2n-1}'(x_{n,k})=0$ such that for continuous $f$ the statement $\|f-H_{2n-1}\|_{L^\infty[-1,1]}\to 0$ as $n\to \infty$ holds.

There has been a large amount of research into analogous statements for weighted polynomial spaces where the interpolation points are chosen to be zeros of certain associated orthogonal polynomials, cf.  Horv\'ath \cite{Hor2007}, \ Lubinsky \cite{Lub1992}, Szabados \cite{Szaba1997,Szaba1999}, Szab\'o \cite{Szabo1999}, and the references therein. For earlier work we refer to the book by Szabados and V\'ertesi \cite{SV1990}.  There are different generalizations of Fej\'er's result; the so called Fej\'er-Hermite interpolation has derivative zero at the interpolation nodes, while the Gr\"unwald operator assigns a derivative value at each node that depends on the function and the location of the nodes. Concretely, the polynomial Gr\"unwald operator is given by
\[
z \mapsto \sum_{k=1}^n f(y_{n,k}) \ell_{n,k}^2(z)
\]
where $(y_{n,k})$ is a given set of nodes and $\ell_{n,k}$ is the $k$th Lagrange interpolating polynomial of degree $\le n$ for $(y_{n,k})$.  As we indicate below, the corresponding operator for functions of exponential type is in some sense the most natural generalization of Fej\'er's result to weighted spaces on the real line.  It was pointed out in Gervais, Rahman, and Schmeisser \cite{GRS1977} that the series
\begin{align}\label{bandlimitedFejer}
F_\tau f(z) = \sum_{k\in\Z} f(\tau^{-1}k) \frac{\sin^2(\tau z)}{(\tau z-k)^2}
\end{align}
has convergence properties entirely analogous to the Fej\'er result but no other entire Gr\"unwald operators seem to have been investigated.

 Let $\nu>-1$. Our first results deals with homogeneous weights
\[
w_\nu(x) = |x|^{2\nu+1}.
\]

Let $J_\nu$ be the Bessel function of order $\nu$ of the first kind.  We define entire functions $A_\nu$ and $B_\nu$ by
\begin{align}\label{ANuTauDef}
\begin{split}
A_\nu(z) &= \Gamma(\nu+1)(z/2)^{-\nu} J_\nu(z),\\
B_\nu(z) &= \Gamma(\nu+1)(z/2)^{-\nu} J_{\nu+1}(z),
\end{split}
\end{align}
and for   $\tau>0$ we define the formal series $G_{\nu,\tau} f, H_{\nu,\tau} f$ by
\begin{align}
\begin{split}
G_{\nu,\tau} f(z) &= \sum_{t\in \mc{T}_\nu} f(\tau^{-1} t) \frac{A_\nu^2(\tau z)}{A_\nu'(t)^2 (\tau z - t)^2}\\
H_{\nu,\tau} f(z) &=  \sum_{t\in \mc{T}_{\nu+1}} f(\tau^{-1} t) \frac{B_\nu^2(\tau z)}{B_\nu'(t)^2 (\tau z - t)^2}
\end{split}
\end{align}
where $\mc{T}_\nu = \{\pm \xi: J_\nu(\xi)=0, \xi > 0\}$.  

Let $f$ be continuous on $\R\backslash\{0\}$ and assume that $fw_\nu$ has a limit at the origin. We say that $fw_\nu$ has a  uniformly continuous extension to $\R$ if, after defining the value of $fw_\nu$ at the origin to be this limit, the resulting function is uniformly continuous on $\R$. 

\begin{theorem}\label{thm1intro} Let $\nu>-1$ with $\nu\neq -\frac12$ and $\tau>0$. If $f\in C(\R\backslash\{0\})$ with $fw_\nu\in L^\infty(\R)$, then $G_{\nu,\tau}f $ and $H_{\nu,\tau} f$ define entire functions of exponential type $2\tau$. If  in addition $fw_\nu$ has a uniformly continuous extension to $\R$ with
\[
\lim_{x\to 0} f(x) w_\nu(x) =0,
\]
then (a) for $\nu>-\frac12$  
\[
\lim_{\tau\to \infty} \|(G_{\nu,\tau}f-f)w_\nu\|_\infty =0,
\] 
and (b) for $-1<\nu<-\frac12$  
\[
\lim_{\tau\to \infty} \|(H_{\nu,\tau}f-f)w_\nu\|_\infty =0.
\]  
 \end{theorem}

 The usual approach to polynomial analogues of \eqref{G-conv}  consists in requiring a condition of the form $fv\in L^\infty(\R)$ with a different weight $v$ which is usually more restrictive than the target weight (but not always, see  Szab\'o \cite[Corollary 2]{Szabo1999}), prove \eqref{G-conv} for a dense set, and extend to the smaller space. In this paper we follow a different approach that is modeled after Fej\'er's original proof. We construct approximations $L_{\nu,\tau}$ to $1/w_\nu$ and use the identity
\begin{align}\label{intro-id}
(G_{\nu,\tau}f -f )w_\nu = (G_{\nu,\tau}f - f L_{\nu,\tau}  w_\nu) w_\nu + fw_\nu( L_{\nu,\tau} w_\nu -1).
\end{align}

A good candidate for $L_{\nu,\tau}$ is the extremal minorant  of $1/w_\nu$ among functions of exponential type $2\tau$ with respect to $L^1(w_\nu)$ norm (cf.\ \cite{CL2017}). We show in Lemmas \ref{lemma2} and \ref{lemma6} that $L_{\nu,\tau} = G_{\nu,\tau} w_\nu^{-1}$ for $\nu>-\frac12$ and $L_{\nu,\tau} = H_{\nu,\tau} w_\nu^{-1}$ for $-1<\nu<-\frac12$ which allows estimation of the first summand on the right. The difference $L_{\nu,\tau} - 1/w_\nu$ has a representation in terms of Laplace transforms that gives $L^\infty$ bounds to control the second summand. 

\medskip

\noindent{\it Remarks.}
\begin{enumerate}
\item Uniform convergence fails for  $f = 1/w_\nu$ if $\nu\neq \frac12$, i.e., the condition $f(x) w_\nu(x)\to 0$ as $x\to 0$ is necessary. (The case $\nu=-\frac12$ is the unweighted case where \eqref{bandlimitedFejer} is used.)

\item To obtain uniform convergence of $(G_{\nu,\tau}f-f)w_\nu$ for $-1<\nu<-\frac12$ requires considerably more restrictive conditions on $f$, and the same remark holds for uniform convergence of $(H_{\nu,\tau}f-f)w_\nu$ for $\nu>-\frac12$. This can be traced back to the fact that the series $G_{\nu,\tau} (1/w_\nu)$ and $H_{\nu,\tau}(1/w_\nu)$ do not minorize  $1/w_\nu$ for these choices of $\nu$.
\end{enumerate}

A second candidate for $L_{\nu,\tau}$ comes from the observation that the space of entire functions $F$ of exponential type $\tau$ with $\| F \|_{L^2(w_\nu)}<\infty$ is a reproducing kernel Hilbert space. Denoting by $K_{\nu,\tau}(w,z)$ the reproducing kernel, it follows from de Branges \cite[Theorem 22]{dB1968} that
\begin{align}\label{introRepKernel}
K_{\nu,\tau}(\bar{z},z)=  \sum_{t\in \mc{T}_\nu} K_{\nu,\tau}(\tau^{-1}t,\tau^{-1}t) \frac{A_\nu^2(\tau z)}{A_\nu'(t)^2 (\tau z - t)^2},
\end{align}
leading to a version of Theorem \ref{thm1intro} with  less general assumptions.

 An identity analogous to \eqref{introRepKernel} holds for every weight $w$ with the property that evaluation functionals are bounded on $\mc{B}_2(\tau,w)$. In particular, reproducing kernels enable us to deal with weights of the form $w(x) = |W(x)|^{-2}$ where $W$ is a Hermite-Biehler entire function of exponential type.  Since the corresponding statements require some notation from the theory of de Branges spaces, we give it as Theorem \ref{thm2} in Section \ref{dBspaces}.

\section{Notation and Bessel function estimates}\label{bessel-section}

Throughout this article $c_\nu, C_\nu$ denote unspecified positive  constants depending only on $\nu$. (Their value may change between lines.) We use the notation   $f(x,\tau) \simeq_\nu g(x,\tau)$ to mean that $f(x,\tau)\le c_\nu g(x,\tau)$ and $g(x,\tau) \le C_\nu f(x,\tau)$ for all $x$ and $\tau$. 

For $\lambda\ge 0$ and complex $z$ we denote by $\g_\lambda$ the Gaussian  
\[
\g_\lambda(z) = e^{-\pi\lambda z^2}.
\]

Bessel functions satisfy $J_\nu^2(x) + J_{\nu+1}^2(x) \simeq_\nu x^{-1}$ for $|x|\ge 1$. Hence  for real $x$
\[
A_{\nu}^2(x) + B_{\nu}^2 (x)\simeq_\nu\begin{cases}
1& \text{ if }|x| \le 1,\\
|x|^{-2\nu-1}& \text{ if }|x| \ge 1.
\end{cases}
\]

 A direct calculation gives
\begin{align*}
A_{\nu}'(z) &= -  \Gamma(\nu+1) (z/2)^{-\nu} J_{\nu+1}( z)\\
A_{\nu}''(z) &= \frac{2^\nu \Gamma(\nu+1) }{z^{2+\nu}}\left( (2\nu+4\nu^2 - z^2) J_\nu( z) - z(1+2\nu) J_{\nu-1}( z)\right)
\end{align*}
and this gives  $|A_{\nu}'(t)| \simeq_\nu | t|^{-\nu -\frac12}$ and $|A_{\nu}''(t)| \simeq_\nu   | t|^{-\nu - \frac32}$  for $t\in \mc{T}_{\nu}$. Analogous statements hold for $B_{\nu}$ and its derivatives, and they lead to the same bounds for $B_{\nu}'(u)$ and $B_{\nu}''(u)$ when $u\in \mc{T}_{\nu+1}$. These estimates are used repeatedly in the calculations below.

\section{The case $\nu>-\frac12$}

\subsection{Extremal minorants} We review the construction of extremal minorants  from \cite{CL2017}. In view of \eqref{homog-lambda-measure} below, we start by constructing entire minorants of exponential type $\tau$ of the Gaussian $\g_\lambda$ with nodes at the zeros of $A_\nu$. It was observed in \cite{CL2017} that this can be achieved by first constructing a minorant of type zero of the exponential $e^{-\lambda z}$ at zeros of $A_\nu(\sqrt{z})$ (see \eqref{Alambda-def} below) and then substituting $z\mapsto z^2$.  In \cite{CL2017} the integration in $\lambda$ is performed on the Fourier transform side to get maximum generality, but since we require $L^\infty$ estimates, we integrate the minorants directly. (The error estimates in Lemma \ref{lemma1} are not contained  in \cite{CL2017}.)

Since the derivative of a  minorant of $1/w_\nu$ must interpolate the derivative of $1/w_\nu$ at any interpolation point, we consider $A_\nu^2$ in the following.   We have
\[
A_\nu^2(z) =\prod_{j=1}^\infty \left(1-\frac{z^2}{\eta_{ \nu,j}^2}\right)^2
\]
where $0<\eta_{\nu,1}<\eta_{\nu,2}<...$  are the positive zeros of $J_\nu$. It is known that $\sum \eta_{\nu,j}^{-2}<\infty$.    Since $A_\nu^2$ is even, the function $F_\nu$ defined by $A_\nu^2(z) = F_\nu(z^2)$ is entire, nonnegative on $\R$, and has only positive (double) zeros $z_j=  \eta_{\nu,j}^2$.  We have $F_\nu(0)=1$ and
\begin{align*}
\sum_{j=1}^\infty z_j^{-1}<\infty. 
\end{align*}

It follows from the theory of Polya-Laguerre entire functions (e.g., \cite[Ch.\ III Corollary 5.4 and Ch.\ V Corollary 3.1]{HW1955}) that $g_\nu$ defined for real $t$ by
\begin{align}\label{gNuTauDef}
g_\nu(t) = \frac{1}{2\pi i } \int_{-i\infty}^{i\infty} \frac{e^{tz}}{F_\nu(z)} dz
\end{align}
is nonnegative, equal to zero on $[0,\infty)$, and satisfies for $\Re z<z_1$ the inversion formula
\[
\frac{1}{F_\nu(z)} = \int_{-\infty}^0 e^{-zt} g_\nu(t) dt.
\]

The inversion formula can be put in the form
\begin{align}\label{exp-rep}
e^{-\lambda z} = F_\nu(z) \int_{-\infty}^\lambda e^{-zt} g_\nu(t-\lambda) dt
\end{align}
in $\Re z<z_1$, and since $g_\nu\ge 0$ on $\R$ it follows that the entire function $\mc{A}_{\lambda,\nu}$ defined by
\begin{align}\label{Alambda-def}
\mc{A}_{\lambda,\nu}(z) = e^{-\lambda z} - F_\nu(z) \int_0^\lambda e^{-wz} g_\nu(w-\lambda) dw
\end{align}
satisfies $\mc{A}_{\lambda,\nu}(x) \le  e^{-\lambda x}$ for all real $x$. Evidently $\mc{A}_{\lambda,\nu}(x) = e^{-\lambda x}$ at the zeros of $F$. Hence, $z\mapsto \mc{A}_{\pi \lambda,\nu}(z^2)$ is a minorant of $\g_\lambda$ on $\R$ that interpolates $\g_\lambda$ at the points of $\mc{T}_\nu$. To simplify notation in the following we set
\[
\xi_1 = \frac14 \eta_{\nu,1}^2,\qquad \xi_2= \frac34 \eta_{\nu,1}^2.
\]

 \begin{lemma}\label{ALambdaNuTauBounds} Let $\lambda>0 $ and $\nu>-\frac12$. There exists $c_\nu>0$ with the following property.
\begin{enumerate}
\item For all complex $z$
\[
\left|e^{-\lambda z} - \mc{A}_{\lambda,\nu}(z) \right| \le c_\nu  |F_\nu(z)| \frac{ e^{-\lambda \xi_1 } - e^{-\lambda \Re z}}{ \Re z- \xi_1 }.
\]

\item In the half plane $\Re z\le \frac12  \eta_{\nu,1}^2$
\[
|\mc{A}_{\lambda,\nu}(z)|\le c_\nu |F_\nu(z)|  e^{-\lambda  \xi_2}.
\]
\end{enumerate}

Moreover,  the function $z\mapsto \mc{A}_{\lambda,\nu}(z^2)$ has exponential type $2$.
\end{lemma}

\begin{proof}  Combining the properties of $g_\nu$ above with \cite[Ch. V Theorem 2.1]{HW1955} implies  that for every $\xi \in (0, \eta_{\nu,1}^2)$ there exists $c_\xi>0$ with
\begin{align}\label{gNuTauEstimate}
\begin{cases}
g_\nu(t) =0\text{ for }t>0,\\
0\le g_\nu(t) \le c_\xi  \exp( \xi t)\text{ for }t<0.
\end{cases}
\end{align}

Applying \eqref{gNuTauEstimate} with $\xi = \xi_1$  in \eqref{Alambda-def} gives for every complex $z$ 
\begin{align*}
\left|\mc{A}_{\lambda,\nu}(z) - e^{-\lambda z}\right| \le c_{\xi_1}  |F_\nu(z)| \int_0^\lambda e^{-w\Re z+  \xi_1 (w-\lambda)} dw
\end{align*}
and  (a) follows.
  
Let now $\Re z< \xi_2$. We replace the term $e^{-\lambda z}$ in \eqref{Alambda-def} by \eqref{exp-rep}, combine the integrals, apply \eqref{gNuTauEstimate} with  $\xi = \xi_2$, and obtain
\[
|\mc{A}_{\lambda,\nu}(z)|\le c_{\xi_2}   |F_\nu(z)|\int_{-\infty}^0 e^{(w-\lambda)(\xi_2-\Re z)} dw.
\]

After evaluation of the integration we further restrict to $\Re z\le \frac12   \eta_{\nu,1}^2$ which leads to (b). Moreover, since $e^{-\lambda z}$ is bounded in $\Re z\ge \frac12  \eta_{\nu,1}^2$, combining (a)  with (b) shows that  $|\mc{A}_{\lambda,\nu }|/(1+|F_\nu |)$  is bounded in $\C$ by a constant (which may depend on $\lambda$ and $\nu$). After substituting $z\mapsto z^2$ the final statement of the lemma follows  since $A_\nu$ has exponential type $1$.
\end{proof}

 For functions that are subordinated to the Gaussian $\g_\lambda(z) = e^{-\pi\lambda z^2}$,  minorants  can now be obtained by integrating an appropriate measure in $\lambda$. We use this to construct an approximation to $1/w_\nu$. The measure is obtained by noting that for $\nu>-\frac12$  and (real) $x\neq 0$
\begin{align}\label{homog-lambda-measure}
|x|^{-2\nu-1} = \pi^{\nu+\frac12} \Gamma\left(\nu+\tfrac12\right)^{-1} \int_0^\infty \g_\lambda(x) \lambda^{\nu-\frac12} d\lambda.
\end{align}

It follows from Lemma \ref{ALambdaNuTauBounds} that in the range $\Re(z^2)\le \frac12 \eta_{\nu,1}^2$ the integral (in $\lambda$) of $\mc{A}_{\pi\lambda,\nu }(z^2)$  with respect to the measure from \eqref{homog-lambda-measure} is convergent, while in the range $\Re(z^2)\ge \frac12 \eta_{\nu,1}^2$ the same is true for the integral of $\mc{A}_{\pi\lambda,\nu }(z^2) -  \g_\lambda(z)$. Since $\g_\lambda$ is also integrable in the latter region with respect to this measure,  the function $z\mapsto L_\nu (z)$ defined by
\[
L_\nu (z) =  \pi^{\nu+\frac12} \Gamma(\nu+\tfrac12) \int_0^\infty \mc{A}_{\pi\lambda,\nu}(z^2) \lambda^{\nu-\frac12} d\lambda
\]
is entire and a minorant of $|x|^{-2\nu-1}$ that interpolates this function at the zeros of $A_\nu $. Furthermore, investigating the bounds obtained from Lemma \ref{ALambdaNuTauBounds}(a) in $\Re z\ge  \frac14 \eta_{\nu,1}^2$ and Lemma \ref{ALambdaNuTauBounds}(b) in $\Re z\le  \frac14  \eta_{\nu,1}^2$ shows that $L_\nu $ has exponential type $2$. To obtain a minorant of type $2\tau$, we define
\[
L_{\nu,\tau}(z) = \tau^{2\nu+1} L_\nu(\tau z).
\]

\begin{lemma}\label{lemma1} Let $\nu>-\frac12$ and $\tau>0$. There exists $c_\nu>0$ so that for real $x$
\[
0\le |x|^{-2\nu-1}-L_{\nu,\tau}(x)  \le c_\nu A_\nu^2(\tau x) \frac{|x|^{-2\nu-1} - \left|\frac12 \tau^{-1} \eta_{\nu,1}\right|^{-2\nu-1} }{\frac14\eta_{\nu,1}^2 -\tau^2 x^2}.
\]

For   (real) $|x| \le \frac1{\sqrt{2}} \tau^{-1} \eta_{\nu,1}$  we also have
\[
0\le L_{\nu,\tau}(x)\le c_\nu A_\nu^2(\tau x) \left(\tfrac{\eta_{\nu,1}}{2\tau}\right)^{-2\nu-1}.
\]
\end{lemma}

\begin{proof} These estimates follow by setting $z=x^2$ in Lemma \ref{ALambdaNuTauBounds}, integrate with respect to $\lambda$ against the measure from \eqref{homog-lambda-measure} to get the bounds for $\tau =1$, multiply by $\tau^{2\nu+1}$ and substitute $x\mapsto \tau x$. 
\end{proof}

\begin{lemma}\label{lemma2} Let $\nu>-\frac12$ and $\tau>0$. The identity
\[
L_{\nu,\tau}(z) = \sum_{t\in\mc{T}_\nu} |\tau^{-1}t|^{-2\nu -1} \frac{A_\nu^2(\tau z)}{A_\nu'( t)^2(\tau z-t)^2}
\]
holds for all complex $z$. 
\end{lemma}

\begin{proof}  It follows from Lemma \ref{ALambdaNuTauBounds}(a) that $x\mapsto \mc{A}_{\pi \lambda,\nu }(x^2)$ is in $L^1(w_\nu)$. The entire function $E_\nu  = A_\nu - i B_\nu$ is Hermite-Biehler, that is, the inequality $|E(z)|>|E(\bar{z})|$ holds for all $\Re z>0$ (cf.\ de Branges \cite[Section 50]{dB1968}). It follows that $z\mapsto \mc{A}_{\pi\lambda, \nu}(z^2)$ satisfies the assumptions of Gon\c{c}alves \cite[Theorem 1]{Gon2017}, see in particular the discussion of homogeneous spaces in \cite[Section 4.1]{Gon2017}. In place of $E_\nu$ we use $iE_\nu$ (which is also Hermite-Biehler) and obtain an interpolation series at the zeros of $A_\nu$.

The minorant property implies that $\g_\lambda(z) - \mc{A}_{\pi\lambda,\nu}(z^2)$ must have derivative equal to zero at $z = t\in\mc{T}_\nu$. Hence we obtain
\begin{align}\label{ASeries}
\begin{split}
\mc{A}_{\pi\lambda,\nu }(z^2) &= \sum_{t\in\mc{T}_\nu } \Bigg(\g_\lambda(t)  \frac{A_\nu^2(z)}{A_\nu '(t)^2(z-t)^2}\left( 1 -\frac{A_\nu ''(t)}{A_\nu'(t)}(z-t)\right) \\
&\qquad \qquad\qquad + \g_\lambda'(t)  \frac{A_\nu^2(z)}{A_\nu'(t)^2(z-t)}  \Bigg).
\end{split}
\end{align}

  Combining estimates from Section \ref{bessel-section}  with the classical fact that the zeros of $J_\nu$ grow at the same rate as the positive integers allows application of  Fubini to interchange summation in $t$ and integration in $\lambda$. It follows that $L_\nu $ has a representation obtained from \eqref{ASeries} by replacing $\g_\lambda(t)$  and $\g_\lambda'(t)$ with $w_\nu(t)^{-1}$
and $\partial/\partial t[w_\nu(t)^{-1}]$, respectively. The differential equation of the Bessel function gives
\[
(2\nu+1)A_\nu'(z) + z A_\nu''(z) = - z A_\nu(z),
\]
which implies that $A_\nu''(t)/A_\nu'(t) = -(2\nu+1)/t$ for $t\in \mc{T}_\nu$. It follows that the  series for $L_\nu$ simplifies to the right hand side of the claimed identity for $\tau =1$, and scaling in $\tau$ gives the general result. 
\end{proof}

\subsection{Gr\"unwald operator}\label{Nu1/2} We give the proof of Theorem \ref{thm1intro}(a). Let $\tau>0$ and $\nu>-\frac12$.  Since $fw_\nu\in L^\infty(\R)$, it follows that the series defining $G_{\nu,\tau}f$ converges uniformly on compact subsets of $\C$ and thus defines an entire function. Since   the difference of consecutive zeros of $A_\nu$ is $\simeq_\nu 1$ we see that $|G_{\nu,\tau}f(z)| \le c_\nu \tau^2 |A_\nu^2(\tau z)|$ and hence $G_{\nu,\tau }f$ has exponential type $2\tau$. (For $|\Im z|\le 1$ we use a contour integral of $(u-z)^{-1}G_{\nu,\tau}f(u)$ over a rectangle with vertical sides through zeros of $A_\nu$ that are at least distance $1$ away from $z$.)

Let $\varepsilon>0$. By assumption there exists $\delta>0$ so that  for $|x-y|<\delta$ and $|u|<\delta$
\begin{align}\label{continuity-assumptions}
\begin{split}
\left|f(y)w_\nu(y) - f(x)w_\nu(x)\right| &<\varepsilon \\
|f(u)w_\nu(u)|&<\varepsilon.
\end{split}
\end{align}

The identity \eqref{intro-id} takes the form
\begin{align}
 \big(&G_{\nu,\tau}(x) - f(x)\big) w_\nu(x) \label{DifferenceExpansion}\\
&= w_\nu(x)\sum_{t\in \mc{T}_\nu} \left(f(\tau^{-1}t)w_\nu(\tau^{-1}t) - f(x)w_\nu(x)\right)  \frac{ w_\nu(\tau^{-1}t)^{-1} A_\nu^2(\tau x)}{A_\nu'(t)^2(\tau x-t)^2} \nonumber\\
&\hspace*{4cm}+ f(x)w_\nu(x) (L_{\nu,\tau}(x)w_\nu(x) -1).\nonumber
\end{align}

 We consider the sum first and partition $\mc{T}_\nu$ into $t$ with $|\tau^{-1}t-x|< \delta$ and  with $|\tau^{-1}t-x|\ge \delta$.  We observe for all $\tau>0$
\begin{align}
 w_\nu(x)\sum_{|\tau^{-1}t-x|<\delta} &\left(f(\tau^{-1}t)w_\nu(\tau^{-1}t) - f(x)w_\nu(x)\right)  \frac{ |\tau^{-1}t|^{-2\nu-1} A_\nu^2(\tau x)}{A_\nu'(t)^2(\tau x -t)^2} \nonumber\\
&\le \varepsilon \,  w_\nu(x) L_{\nu,\tau}(x)\le \varepsilon.\label{1aSumPlus} 
\end{align}

For the second sum we use that $fw_\nu$ is uniformly bounded on $\R$ (by $M$, say), and we obtain
\begin{align}\label{Mbound}
\begin{split}
w_\nu(x) \sum_{|\tau^{-1}t-x|\ge \delta} \big(&f(\tau^{-1}t)w_\nu(\tau^{-1}t) - f(x) w_\nu(x)\big)  \frac{ |\tau^{-1}t|^{-2\nu-1} A_\nu^2(\tau x)}{A_\nu'(t)^2(\tau x-t)^2}\\
&\le 2M w_\nu(x)  A_\nu^2(\tau x) \sum_{|\tau^{-1}t-x|\ge \delta}  \frac{ |\tau^{-1}t|^{-2\nu-1}}{A_\nu'(t)^2(\tau x-t)^2}
\end{split}
\end{align}
 
Combining the estimates of Section \ref{bessel-section} with $w_\nu(x) \le w_\nu(1/\tau)$ for  $|x|\le 1/\tau$  gives $w_\nu(x) A_\nu^2(\tau x) \le c_\nu w_\nu(1/\tau)$ for all $x$ and $\tau$.  Denote by $t_+$ the zero of $A_\nu$ greater  (smaller) than $t$ if $t$ is positive (negative).   Since  $1\le c_\nu |t-t_+|$ for all   zeros of $A_\nu$, the final expression in \eqref{Mbound} is 
\begin{align}\label{1aSumMinus}
\le c_\nu  M   \frac{1}{\tau} \sum_{\substack{ t\in \mc{T}_\nu \\ |\tau^{-1}t-x|\ge \delta}} \frac{|\tau^{-1}t-\tau^{-1}t_+|}{( x-\tau^{-1}t)^2}\le  \frac{c_\nu  M }{\tau \delta}
\end{align}
where the last inequality above follows by recognizing the series as a Riemann sum for the integral of $u\mapsto (x-u)^{-2}$ on $\R\backslash[x-\delta,x+\delta]$. 

It remains to analyze $w_\nu  f  ( L_{\nu,\tau} w_\nu -1)$. It follows from Lemma \ref{lemma1} that 
\begin{align*}
|L_{\nu,\tau}(x)w_\nu(x) -1|\le c_\nu A_\nu^2 (\tau x) w_\nu(x) \frac{\left| w_\nu(x)^{-1} - w_\nu(\frac{\eta_{\nu,1}}{2\tau})^{-1}\right|}{\frac{\eta_{\nu,1}^2}{4} - \tau^2 x^2} 
\end{align*}
which implies that $L_{\nu,\tau}(x)w_\nu(x) -1$ converges to zero uniformly for $|x|\ge \delta$ as $\tau\to \infty$. For $|x|\le \delta$ we use $0\le L_{\nu,\tau} w_\nu\le 1$. Combining this with \eqref{continuity-assumptions},  \eqref{1aSumPlus}, and \eqref{1aSumMinus} leads to
\[
\limsup_{\tau\to \infty} \| (G_{\nu,\tau} f -f) w_\nu\|_\infty < 2\varepsilon,
\]
and hence the claim of Theorem \ref{thm1intro}(a).

\section{The case $-1<\nu<-\frac12$}

\subsection{Extremal minorants} For the construction of the minorant of $1/w_\nu$ with $-1<\nu<-\frac12$ only a few modifications need to be made. We have
\[
B_\nu(z) = z \prod_{j=1}^\infty \left(1-\frac{z^2}{\eta_{\nu+1,j}^2}\right)
\]
where $0<\eta_{\nu+1,1}<\eta_{\nu+1,2}<...$ are all positive zeros of $J_{\nu+1}$. We define $F_\nu$ by  $B_\nu^2(z) = F_\nu(z^2)$, and we observe that $F_\nu$ is entire, nonnegative on $(0,\infty)$, negative on $(-\infty,0)$ with a simple zero at the origin, double zeros at $z_j=\eta_{\nu+1,j}^2$, and no other zeros in $\C$. We set
\[
g_\nu(t) = \frac{1}{2\pi i} \int_{-1-i\infty}^{-1+i\infty} \frac{e^{tz}}{F_\nu(z)} dz
\]
and we note that $g_\nu$ is nonpositive on $\R$, equal to zero on $[0,\infty)$ and satisfies \eqref{exp-rep} for $\Re z<0$. Defining $\mc{A}_{\lambda,\nu}$ by \eqref{Alambda-def}, the sign of $g_\nu$ gives
\[
\mc{A}_{\lambda,\nu}(x^2)\ge \g_\lambda(x)
\]
for real $x$.  (Since $F_\nu$ has a zero at the origin for $-1<\nu<-\frac12$, the  function $g_\nu$ is bounded on $t<0$, but does not converge to zero as $t\to -\infty$.)  An integration by parts shows that the two-sided Laplace transform of $g_\nu'(t)$ equals the reciprocal of $F_\nu(z)/z$ in $\Re z<\eta_{\nu+1,1}^2$ (hence in particular $g_\nu'\ge 0$), and it follows that $g'$ satisfies \eqref{gNuTauEstimate} for $\xi<\eta_{\nu+1,1}^2$. We set
\[
 \xi = \frac14\eta_{\nu+1,1}^2
\]
in the following two lemmas.

\begin{lemma}\label{AGrowthNu2} Let $-1<\nu<-\frac12$ and $\lambda>0$. Then $\mc{A}_{\lambda,\nu}$ is an entire function of exponential type $2$ that satisfies the following growth estimates.
\begin{enumerate}
\item For all complex $z$
\[
\left|\mc{A}_{\lambda,\nu}(z) - e^{-\lambda z}\right| \le  \frac{|F_\nu(z)|}{|z|} \left(  |g_\nu(-\lambda)|+     \int_0^\lambda e^{-w \Re z + (w-\lambda)\xi} g_\nu'(w) dw\right)
\]
\item For $\Re z\le 0$
\[
|\mc{A}_{\lambda,\nu}(z)| \le -\frac{|F_\nu(z)|}{|z|} |g_\nu(-\lambda)| + \frac{e^{-\lambda  \xi}}{|z - \xi|}
\]

\end{enumerate}
\end{lemma}

\begin{proof}   An integration by parts shows 
\begin{align*}
\mc{A}_{\lambda,\nu}(z) -e^{-\lambda z} &= -\frac{F_\nu(z)}{z} g_\nu(-\lambda) -\frac{F_\nu(z) }{z} \int_0^\lambda e^{-wz} g_\nu'(w-\lambda) dw
\end{align*}
  which gives  (a). Combining with \eqref{exp-rep} gives for $\Re z < \xi$
\begin{align}\label{A-rep-neg}
\mc{A}_{\lambda,\nu}(z) = \frac{F_\nu(z)}{z} \left(-g_\nu(-\lambda) + \int_{-\infty}^0 e^{-zt} g_\nu'(w-\lambda) dw\right)
\end{align}
leading to (b). As in the proof of Lemma \ref{ALambdaNuTauBounds} it follows that $z\mapsto \mc{A}_{\lambda,\nu}(z^2)$ is entire and has exponential type $2$.
\end{proof}

 We require sharper estimates on the real line.

\begin{lemma}\label{lemma5} For $-1<\nu<-\frac12$, $\lambda>0$, and $x>0$
\[
0\le \mc{A}_{\lambda,\nu}(x) - e^{-\lambda x} \le c_\nu \frac{F_\nu(x)}{x} \left(\frac{1-e^{-\lambda  \xi}}{\xi} - \frac{e^{-\lambda x} - e^{-\lambda  \xi}}{ \xi -x}\right)
\]
\end{lemma}
\begin{proof} We have for $x>0$ 
\begin{align*}
\mc{A}_{\lambda,\nu}(x) -e^{-\lambda x} &= \frac{F_\nu(x)}{x} \int_{-\lambda}^0 (1- e^{-x(u+\lambda)}) g_\nu'(u) du.
\end{align*}
An application of \eqref{gNuTauEstimate} for $g'$ gives the upper inequality.
\end{proof}

The identity
\begin{align}\label{homog-lambda-measure-2}
|x|^{-2\nu-1} = \pi^{\nu+\frac12}\left(- \Gamma\left(\nu+\tfrac12\right)\right)^{-1} \int_0^\infty (1-\g_\lambda(x)) \lambda^{\nu-\frac12} d\lambda,
\end{align}
valid for  $-1<\nu<-\frac12$  and (real) $x\neq 0$, suggests the definition
\[
L_\nu(z) = \pi^{\nu+\frac12} \left(-\Gamma\left(\nu+\tfrac12\right)\right)^{-1}   \int_0^\infty (1-\mc{A}_{\pi\lambda,\nu}(z^2)) \lambda^{\nu-\frac12} d\lambda,
\]
and as before, $L_{\nu,\tau}(z) = \tau^{2\nu+1} L_\nu(\tau z).$

\begin{lemma}\label{lemma6} Let $-1<\nu<-\frac12$.  The integral defining $L_\nu$ converges uniformly on compact subsets of $\C$ and is an entire function of exponential type $2 $. We have for real $x$
\begin{align*}
0&\le |x|^{-2\nu-1} - L_{\nu,\tau}(x)\\
& \le \frac{B_\nu^2(\tau x)}{(\tau x)^2} \left( \frac{|\frac12 \tau^{-1} \eta_{\nu+1,1}|^{-2\nu-1}}{\frac14 \eta_{\nu+1,1}^2} - \frac{|x|^{-2\nu-1} - |\frac12 \tau^{-1}\eta_{\nu+1,1}|^{-2\nu-1}}{\frac14 \eta_{\nu+1,1}^2 - \tau^2 x^2}\right).
\end{align*}
\end{lemma}

\begin{proof} Writing $g_\nu(-\lambda)$ as an integral of its derivative on $[-\lambda,0]$ and observing the growth estimates for $g_\nu'$  shows that $g_\nu(-\lambda) \lambda^{\nu-\frac12}$ is integrable on $[0,\infty)$. Similarly, the integral in Lemma \ref{AGrowthNu2}(a) is integrable with respect to $\lambda^{\nu-\frac12} d\lambda$, which gives the first part of the lemma. This also gives the necessary estimates in $\Re z\ge -1$ to show the claimed exponential growth in this half plane.

Let $\Re z\le -1$. Equation \eqref{A-rep-neg} combined with the inversion formula for $g_{\nu,\tau}'$ gives 
\begin{align*}
1-\mc{A}_{\lambda,\nu}(z) = \frac{F_\nu(z)}{z} \left(g_\nu(-\lambda) + \int_{-\infty}^0 e^{-zt} \big(g_\nu'(w) -g_\nu'(w-\lambda) \big) dw\right).
\end{align*}

Since $g_\nu'(w-\lambda) - g_\nu'(w) = \mc{O}(1-e^{\lambda w})$, the integral can be integrated in $\lambda$ with respect to the measure $\lambda^{\nu-\frac12} d\lambda$, and the growth estimates in $\Re z\le -1$ follow as above. We leave the details to the reader. The minorant property follows from $\mc{A}_{\lambda,\nu}(x^2)\ge G_\lambda(x)$ and $\Gamma(\nu+\frac12)<0$.

The last inequality of the lemma follows by integrating the inequality of Lemma \ref{lemma5} in $\lambda$, substituting $x^2$ for $x$, observing \eqref{homog-lambda-measure-2}, and scaling by $\tau$.
\end{proof}

\begin{lemma}\label{lemma7} Let $-1<\nu<-\frac12$ and $\tau>0$. The identity
\[
L_{\nu,\tau}(z) = \sum_{t\in\mc{T}_{\nu+1}} |\tau^{-1}t|^{-2\nu -1} \frac{B_\nu^2(\tau z)}{B_\nu'(t)^2(\tau z-t)^2}
\]
holds for all complex $z$. 
\end{lemma} 

\begin{proof}  Denote by $\Gamma_t$ the square in $\C$ with center at the origin and sides through $\pm  t$. Since $-1<\nu<-\frac12$, it follows that $1/B_\nu(t)\to 0$ if $|t|\to \infty$ through the points $t$ with $J_\nu( t) =0$. Moreover, for fixed $t$ we have $1/|B_\nu(t+iy)|\le 1/|B_\nu(t)|$, and  we have $1/|B_\nu(x+iy)|\to 0$ uniformly in $x$ as $|y|\to \infty$.  Integrating $(w-z)^{-1} B_\nu^{-2}(w)$  in $w$ over $\Gamma_t$ with $t\in \mc{T}_{\nu+1}$, applying the residue theorem, and letting  $|t|\to \infty$ shows that the regular part in the representation of $1/B_\nu^2$ obtained from the Mittag-Leffler theorem is equal to zero.

We combine this with the fact from \cite{Gon2017} that the function $z\mapsto \mc{A}_{\lambda,\nu }(z^2)$ has an analogous representation to \eqref{ASeries} with $\mc{T}_\nu$ replaced by $\mc{T}_{\nu+1}\cup\{0\}$ (this set contains the zeros of $B_\nu$) and $A_\nu$ replaced by $B_\nu$. It follows that
\begin{align}
1-\mc{A}_{\pi\lambda,\nu}(z^2) &= \sum_{t\in\mc{T}_{\nu+1}} \Bigg((1 - \g_\lambda(t))  \frac{B_\nu^2( z)}{B_\nu'(t)^2( z-t)^2}\left( 1 -\frac{B_\nu''(t)}{B_\nu'(t)}(z-t)\right)\nonumber\\
&\qquad \qquad\qquad - \g_\lambda'(t)  \frac{B_\nu^2(z)}{B_\nu'(t)^2(z-t)}  \Bigg).\label{BSeries}
\end{align}

The remaining steps utilize the differential equation
\[
 B_\nu''(z) +  \frac{2\nu+1}{ z} B_\nu'(z) + \left(1+\frac{\nu^2 - (\nu+1)^2}{ z^2}\right) B_\nu(z) =0
\]
and proceed analogously to the proof of Lemma \ref{lemma2}.
\end{proof}

\subsection{Gr\"unwald operator} The proof of Theorem 1(b) proceeds completely analogously to the arguments in Section \ref{Nu1/2} by replacing $\mc{T}_\nu$ with $\mc{T}_{\nu+1}$, the function $A_\nu$ by $B_\nu$, observing nonnegativity of $L_{\nu,\tau}$ from Lemma \ref{lemma7}, and estimating $fw_\nu(L_{\nu,\tau} w_\nu-1)$ using Lemma \ref{lemma6}. We omit the details.

\section{De Branges spaces}\label{dBspaces}

 For $1\le p\le \infty$ we denote by $H^p(\C^+)$ the space of analytic functions $F$ in the upper half plane $\C^+$ for which $\sup_{y>0} \|F(.+iy)\|_p$ is finite. We require a few facts from the theory of de Branges spaces. For a more complete picture we refer to de Branges \cite{dB1968} for $p=2$ and Baranov \cite{Bar2003} for arbitrary $p$.

 An entire function $E$ satisfying
\begin{align}\label{HBdef}
|E(z)|>|E(\bar{z})|
\end{align}
for $\Im z>0$ will be called a Hermite-Biehler function.  Throughout this article we assume that $E$ has no real zeros. We set
\[
\mc{H}^p(E) = \{F\text{ entire}: F/E, F^*/E \in H^p(\C^+)\},
\]
where $F^*$ is the entire function $F^*(z) = \overline{F(\bar{z})}$, and the norm is given by $F\mapsto \|F/E\|_p$.  We write $E = A-iB$ with real entire $A = 2^{-1}(E+E^*)$ and $B=2^{-1}i(E-E^*)$, and we denote by $\mc{T}_A$ the set of zeros of $A$. These are necessarily real by \eqref{HBdef}. We denote by $K_E$ the function
\begin{align}\label{repKernel}
K_E(w,z) = \frac{A(\bar{w})B(z) - A(z) B(\bar{w})}{\pi(z-\bar{w})},
\end{align}
and we observe that $z\mapsto K_E(w,z)$ is an entire function in $\mc{H}^p(E)$ for all $w\in\C$ and $p\in[0,\infty]$. The Cauchy integral formula for the upper half plane and the alternative representation $K_E(w,z) = [2\pi i(z-\bar{w})]^{-1}(E(\bar{w}) E^*(z) - E(z) E^*(\bar{w}))$ may be used to check that $K_E(w,z)$ is the reproducing kernel for $\mc{H}^2(E)$. A function $\varphi_E$ with the property $E(x) e^{i\varphi_E(x)}\in\R$ for all real $x$ is called a {\it phase} of $E$. As a consequence of \eqref{HBdef} it may be chosen to be analytic on an open set containing the real line.

If $\mc{B}_2(\tau,w)$ has bounded evaluation functionals, then by \cite[Theorem 23]{dB1968} it is isometrically equal to a space $\mc{H}^2(E)$.

Let $f$ be continuous and $f|E|^{-2}\in L^\infty(\R)$. We define the formal series $G_E f$ by
\[
G_E f(z) = \sum_{t\in \mc{T}_A} f(t) \frac{A^2(z)}{A'(t)^2(z-t)^2}.
\]

In the following $\{E_\tau:\tau>0\}$ is a collection of Hermite-Biehler functions $E_\tau$ of exponential type $\tau$.  We set $G_\tau f = G_{E_\tau} f$, $\mc{T}_\tau=\mc{T}_{A_\tau}$, and $\varphi_\tau = \varphi_{E_\tau}$ if the choice of $E_\tau$ is clear from the context. We consider $f$ with the following  properties.
\begin{enumerate}
\item For every $\varepsilon>0$ there exists $\delta>0$ and $\tau_0>0$ such that for all $x,t\in\R$ and all $\tau\ge \tau_0$
\begin{align}\label{f-ucont}
|x-t|<\delta\qquad \Longrightarrow \qquad \left| \frac{f(x)}{|E_\tau(x)|^2} - \frac{f(t)}{|E_\tau(t)|^2}\right|<\varepsilon
\end{align}

\item There exists $M>0$ so that for all $x\in\R$ and positive $\tau$
\begin{align}\label{f-bound}
\frac{f(x)}{|E_\tau(x)|^2} \le M
\end{align}
\end{enumerate}

If $|E_\tau|^{-2}$ converges to $w$  uniformly in $x$ as $\tau\to \infty$, then these conditions may be reformulated in terms of uniform continuity and uniform boundedness of $fw$.

\begin{theorem}\label{thm2} Let $E_\tau$, $\tau>0$, be a Hermite-Biehler entire function of exponential type $\tau$ without real zeros, and such that there exists $C>0$ so that $|\varphi_\tau'(x) - \tau|\le C$ for all $x$ and $\tau$. 

\begin{enumerate} 
\item If  $f:\R\to\C$ satisfies \eqref{f-ucont} and \eqref{f-bound}, then $G_\tau$ converges uniformly on compact sets, $G_\tau\in \mc{H}^\infty(E_\tau^2)$ and 
\[
 \lim_{\tau\to\infty} \| (G_\tau f -f)E_\tau^{-2}\|_\infty =0.
\]
\item If in addition there exists constants $c>0, d>0$ with  
\[
c \leq |E_\tau(x)|^2w(x) \leq d
\]
for all real $x$ and $\tau\ge \tau_0$, then 
\[
 \lim_{\tau\to\infty} \| (G_\tau f -f)w\|_\infty =0.
\]\label{partb-intro-thm1}
\end{enumerate}
\end{theorem}

\begin{proof} Let $\varepsilon>0, \delta>0, \tau_0>0$ as in \eqref{f-ucont}.  We apply \cite[Theorem 22]{dB1968} with $iE$ in place of $E$ to $z\mapsto  K_\tau(x,z)$ with $x\in\R$ to obtain
\[
K_\tau(x,x) = \|K_\tau(x,.)/E\|_2^2 = \sum_{t\in \mc{T}_\tau} \frac{|K_\tau(x,t)|^2}{K_\tau(t,t)}.
\]

Equation \eqref{repKernel} leads to
\begin{align}\label{weightedFejerIdentity}
K_\tau(x,x) = \sum_{t\in \mc{T}_\tau} K_\tau(t,t) \frac{A_\tau^2(x)}{A_\tau'(t)^2(x-t)^2}.
\end{align}

It follows that
\begin{align*}
\frac{f(x) - G_\tau(x)}{K_\tau(x,x)} &= \frac{1}{K_\tau(x,x)} \sum_{t\in \mc{T}_\tau} \left( \frac{f(x)}{ K_\tau(x,x) } - \frac{f(t)}{K_\tau(t,t)} \right) \frac{K_\tau(t,t) A_\tau^2(x)}{A_\tau'(t)^2(x-t)^2}. 
\end{align*}

We recall the identity $\pi K_\tau(x,x) = \varphi_\tau'(x) |E_\tau(x)|^2$ from \cite[Problem 48]{dB1968}. In the range $|x-t|\ge \delta$ we use $\pi K_\tau(t,t) A_\tau'(t)^{-2} = (\pi \varphi_\tau'(t))^{-1}\le c \tau^{-1}$. In the range $|x-t|<\delta$ we add and subtract $f(t) (\varphi_\tau'(x) |E_\tau(t)|^2)^{-1}$ and use $|\varphi_\tau'(x)^{-1} - \varphi_\tau'(t)^{-1}|\le c\tau^{-2}$ with $c$ independent of $t$ and $x$. Since the values in $\mc{T}_\tau$ are the points where $E_\tau$ is purely imaginary we have $\pi = \varphi_\tau(t) - \varphi_\tau(s)$ for consecutive $s,t\in\mc{T}_\tau$, and the mean value theorem gives $\tau^{-1} \le c(s-t)$.  With these estimates (a) can be proved analogous to \eqref{1aSumPlus}, \eqref{Mbound} and \eqref{1aSumMinus}, and we leave the details to the reader. Statement (b) is an immediate consequence.
\end{proof}

Let $\tau\ge \tau_0\ge 0$, and assume that $W$ is a Hermite-Biehler entire function of exponential type $\tau_0$. Define $E_\tau(z) = e^{(\tau-\tau_0)z} W(z)$ and real entire $A_{\tau,\alpha}, B_{\tau,\alpha}$ by $e^{i\alpha} E_\tau = A_{\tau,\alpha} - i B_{\tau,\alpha}$. Define formal series $G_{\tau,\alpha} f$ by
\[
G_{\tau,\alpha}f(z) = \sum_{t\in \mc{T}_{\tau,\alpha}} f(t) \frac{A_{\tau,\alpha}^2(z)}{A_{\tau,\alpha}'(t)^2(z-t)^2},
\]
where $\mc{T}_{\tau,\alpha}$ is the set of (real) zeros of $A_{\tau,\alpha}$. A direct consequence of Theorem \ref{thm2} is the following statement.

\begin{corollary} Let $f|W|^{-2}\in L^\infty(\R)$ be uniformly continuous. Then $G_{\tau,\alpha}f$ defines an entire function of exponential type $2\tau$, and
\[
\lim_{\tau\to\infty} \| (G_{\tau,\alpha}f - f)|W|^{-2}\|_\infty =0.
\]

\end{corollary}

This corollary includes for example the Poisson measure $dx/(1+x^2)$. 

\medskip

We finally describe some open questions.  It would be of interest to have a characterization of those entire functions that may be used in place of $A_\tau$ and give uniform convergence of the corresponding Gr\"unwald operator, but even for simple measures this seems out of reach.  We consider the concrete example $w(x) = x^2+1$. First, for  the Hermite-Biehler entire function
\begin{align}\label{we-def}
E_\tau(z) = \left(\frac{2}{\sinh(2\tau)}\right)^{\frac12} \frac{\sin(\tau(z+i))}{z+i}
\end{align}
we observe that $\tanh(\tau) \le (x^2+1)|E_\tau(x)|^2 \le \coth(\tau)$. Since
\[
K_\tau(x,x) = \frac{1}{\pi}\left(\frac{\tau}{(x^2+1)|E_\tau(x)|^2} - \frac{1}{x^2+1}\right)|E_\tau(x)|^2,
\]
we obtain 
\begin{align*}
\varphi'(x) &= \tau \left(\frac{\text{sinh}(2\tau)}{\text{cosh}(2\tau) - \cos(2\tau x)}\right) - \frac{1}{x^2+1},\\
A_\tau(z) &=  \sqrt{\frac{2}{\text{sinh}(2\tau)}} \frac{z\text{cosh}(\tau)\sin(\tau z) + \text{sinh}(\tau)\cos(\tau z)}{z^2 + 1}.
\end{align*}

Hence for $f$ with $fw\in L^\infty(\R)$ and uniformly continuous on $\R$ the interpolation series $G_{E_\tau} f$ satisfies $\| (G_{E_\tau} f - f)w\|_\infty\to 0$. This is true in particular for $f = 1/w$. 

On the other hand, it can be shown that starting with an even Polya-Laguerre function $A$ of exponential type $1$ with $A(0)=1$, $A'(t)\le C/t$  and $|s-t|\ge D>0$ for all zeros $s,t$ of $A$ the dilation $A_\tau(z) = A(\tau z)$ has the property that $G_{A_\tau} (1/w)(0)\to\infty$ as $\tau\to\infty$. In particular, using the dilation $E_1(\tau z)$ in place of \eqref{we-def} fails to give a uniformly converging interpolation.

In contrast, $A_\tau(z) = \cos(\tau z)$ has the property that $\|(G_{A_\tau} f-f)w\|_\infty\to0$ as $\tau\to \infty$ for $f$ with $fw\in L^\infty(\R)$ that is uniformly continuous on $\R$; the series
\[
G_{A_\tau} w^{-1}(z) = \sum_{\cos(\tau t)=0} \frac{1}{t^2+1} \frac{\cos^2(\tau z)}{\tau^2(z-t)^2}
\]
has a closed form that may be used to show $w(x) G_{A_\tau} w^{-1}(x) = 1+ \tau^{-1} p_\tau(x)$ with $|p_\tau(x)|\le2$ for real $x$, which allows an argument analogous to the proofs of Theorems \ref{thm1intro} and \ref{thm2}.

\bibliographystyle{amsplain}

\bibliography{references}

\end{document}